\documentclass{amsart}                 % final size
 
 \usepackage{graphicx}
  \usepackage[format=plain,font=small]{caption}
 
\usepackage{amssymb,amsmath,amsfonts}
 
\usepackage{epsfig} 
   \usepackage{pdfsync}
 \usepackage[usenames]{color}
\definecolor{MyDarkBlue}{rgb}{0,0.08,0.45} 
\definecolor{MyDarkRed}{rgb}{.8,0.09,0.5} 

\definecolor{Red}{rgb}{1,0,0} 
\definecolor{Blue}{rgb}{0,0,1}

\def\R{{\mathbb R}}
\def\C{{\mathbb C}}

\def\Z{{\mathbb {Z}}}

\def\cF{{ \mathcal F}}

 \def\fer{{g}}

\def\rf{{r}}

 \def\snc{{\rm sinc}}
\newcommand\what[1]{\widehat{#1}}
\def\minboundu{ \mathop {\alpha_{11}}}
\newcommand\maxboundu{\mathop  {\beta_{11}}}
\newcommand\minboundd{\mathop  {\alpha_{22}}}
\newcommand\maxboundd{\mathop {\beta_{22}}}

\def\PP{{P}}
\def\Ld{{L^2(\R)}} %L^2 di R
\def\Lu{{L^1(\R)}} %L^2 di R
 \def\S{{R}}
  \def\ldCN {{{\ell}^2 (\Z;{\C}^N)}} %Successioni
\def \Tfi{{T^{}_{\Phi,\hskip 0.05em t_o}}}
\def\Tfia{{T^{*}_{\Phi,\hskip 0.05em t_o}}}

\def\Tfistar{{T^{}_{\Phi^*,\hskip 0.05em t_o}}}

\def\Tfistara{{T^{*}_{\Phi^*,\hskip 0.05em t_o}}}

\def\Efi{{E_{\Phi,\hskip 0.05em t_o}}}
\def\EPfi{{E_{P\Phi,\hskip 0.05em t_o}}}

\def\Efistar{{E_{\Phi^*,\hskip 0.05em t_o}}}

\def\BL{{B_\omega}}
 \newcommand\cU {\mathop{ \mathcal U}}
 \newcommand\cI {\mathop{ \mathcal I}}
  
   \newcommand\Sfi {S }
  
 \newcommand\Cfi {C}
   
\def\squareforqed{\hbox{\rlap{$\sqcap$}$\sqcup$}}
\def\qed{\ifmmode\squareforqed\else{\unskip\nobreak\hfil
\penalty50\hskip1em\null\nobreak\hfil\squareforqed
\parfillskip=0pt\finalhyphendemerits=0\endgraf}\fi} 
 \newtheorem{theorem}{Theorem}[section] 
\newtheorem{corollary}[theorem]{Corollary}

\newtheorem{proposition}[theorem]{Proposition}

\theoremstyle{remark}

\numberwithin{equation}{section}

\begin{document}
 
\title[Recovery of missing  samples]
{ Stability of the recovery  of Missing \\ Samples   in   Derivative Oversampling }

\subjclass[2000]{} 

 
 \author{Paola Brianzi and Vincenza Del Prete  }

  \address{Dipartimento di Matematica\\
 Universit\`a di Genova, via Dodecaneso 35, 16146 Genova \\ Italia}
\begin{abstract}     This paper deals 
 with  the    problem of  reconstructing a band-limited signal when a  finite subset of   its samples and of  its derivative are missing. The technique  used, due  to P.J.S.G. Ferreira,  is based on the  use of a particular frame for band-limited functions and the relative oversampling formulas.  We study the eigenvalues of the matrices  arising in the  procedure  of   recovering  the  lost samples,  finding    estimates of their eigenvalues  and their dependence on   the oversampling parameter  and on the number of missing  samples. 
 When the missing samples are consecutive, the problem may become very ill-conditioned. We present a numerical procedure      to  overcome this difficulty, also in presence of noisy data, by using Tikhonov regularization techniques.  
  \end{abstract}
 
\maketitle
\setcounter{section}{0}

  \section{Introduction} \label{s:Introduction}
This paper deals with the  problem of  recovering missing samples of band-limited functions and  of its derivative via  frame reconstruction sampling formulas.  A band-limited signal is a function which belongs to the  space $\BL$  of functions in $\Ld$ whose Fourier transforms have support in $[-\omega,\omega]$. 
 Functions in  this   space can be   expanded     in terms of the  
orthonormal basis of translates of the {\it sinc} function. The  
coefficients of the expansion   are the  samples of the function  at  a 
uniform grid on $\R$,  with  ``density" $ {\omega}/{\pi}$ ({\it Nyquist} 
density).  Such expansions, called  sampling formulas,    have been 
generalized by replacing the orthonormal basis with  more general 
families, like  Riesz bases and frames, formed by the translates of 
one or more functions ({\it generators}). Frames, unlike Riesz basis, 
are overcomplete and  their  redundancy  provides a perfect tool  
for the recovery of missing data. \par 
  The problem of recovering missing samples   has been 
investigated  first by  P.J.S.G. Ferreira in  \cite{F}, where it is shown   
that,  in the case of  a generalized Kramer sampling,   under 
suitable oversampling assumptions, any  finite   set of  missing samples can be recovered from the others by solving a  linear system   $(I-S)x=b$, where the matrix $S$ is positive definite. Moreover  the author studies  the eigenvalues of 
the matrix $S$ in dependence of the  oversampling parameter  and 
the number of missing samples. Successively    D.M.S. Santos and Ferreira    considered   the case of   a   two-channel  derivative  oversampling formula  obtained  by 
projecting the generators of a   Riesz  basis   of  the space $\BL$    and their  
duals    into   the space  $B_{\omega_a} $ with $\omega_a<\omega$   \cite{SF}.
With this technique  they  obtain  a pair of dual frames, 
  although  the dual frame is not  the  
canonical one. In  their paper the authors show that      a finite number of
missing samples either of the function   or of  its derivative can be
recovered, solving in each case a  non singular     linear system  
similar to the one-channel case. However the authors do not   consider  the case  when  samples both of the function and the derivative are  simultaneously missing.   
\\  The technique of Santos and Ferreira     has been applied
to more general two-channels by J. M. Kim and K. H. Kwon \cite{KK}, 
who gave sufficient conditions for the    recovery of  missing samples,   from a single channel and  from both channels. However,   as the   authors observe,  these conditions do  not apply to    the derivative sampling formulas studied in \cite{SF}.  
\par 
Successively, in  \cite{DP}, the author 
gave  new derivative    oversampling  formulas of any order, 
where  the  dual  generator is  the canonical dual  and in \cite{DP1} 
found an  expression
 of the dual in the two-channel  case.  The use of the canonical 
dual  allowed the author to prove in \cite{DP1} that simultaneously 
missing samples of the  function  and of the derivative can be 
recovered.  However the generators of the canonical duals are 
much less explicit than the non-canonical dual used by Santos and 
Ferreira see  \cite[(5.2) p. 178]{DP1} and (\ref{antif}).\\
 In this paper we   study the problem of the recovery of missing  samples  of the  function  and of its derivative    
in Ferreira's  two-channel  derivative  formula. The  technique  we use, first proposed in     \cite{SF},      extends in a natural  way 
that  for  one channel   in \cite{F}  and  consists in 
solving   a system $(I-\Sfi)X=\Cfi$, 
 where the    matrix   $\Sfi$   (see (\ref{Sfi})) is a block matrix depending on the dual generators  and   on the position   of the missing samples, and the  unknowns $X$ are the missing samples of the  function and its derivative. To the best of our knowledge, no results are known about the possibility of solving  this system, i.e. of recovering simultaneous missing samples of the function and of its derivative.   \\
 We obtain estimates of the minimum and maximum eigenvalues of 
the block submatrices  of the matrix $\Sfi$, show that in  some 
cases the matrix  reduces to a  lower triangular matrix and provide 
several numerical 
experiments  on   the dependence  of the eigenvalues   on    the 
parameters  of the problem, like   the oversampling parameter $\rf$ and  the number of missing samples.  Moreover, we  experiment the recovery of  missing samples  and the reconstruction of a signal,  paying  particular attention to the computational aspects,  analyzing  the ill-conditioned  problem of    contiguous samples,   also  with noisy   data.\\
  The paper is organized as follows. In Section 2  we establish 
notation  and collect some  results   to be used  later.  In Section 3  we  describe the technique  for   the recovery of the  missing samples.
  Section 4 is dedicated to  the study of the 
eigenvalues  of the  coefficient   matrices in the  system,  mostly   when the missing samples are  equidistant. In this case,  we find  estimates of the minimum and maximum eigenvalues   of  the  
diagonal submatrices $S_{11}$ and $S_{22}$ of $\Sfi$  (see (\ref{Sfi})). Moreover,  we find that in some cases   the matrix $S$ is   lower triangular   and that half of its eigenvalues are equal to $2\rf-\rf^2$, half to   $\rf^2$, where $\rf$ is the oversampling parameter. This extends to two channels  a  result  in \cite{F}.  
We also present some    numerical experiments  supporting the theory.  Finally, in Section 5, we  
   analyze   the   case of contiguous missing samples   both for one and two channels, when the problem may become   very    ill-conditioned.  We present a numerical procedure   to solve it, also in presence of   noisy data, via regularization techniques.  
 
 \section{Preliminaries} \label{s:Preliminaries}
 In this section  we   collect some  results on frames       to be used later. We begin by introducing some notation. 
The Fourier transform of a  function  $f$ in 
  $  \Lu $ is
 $$ \cF f(\xi)={\hat f}(\xi)= \frac{1} {\sqrt{2\pi}} \int f(t)e^{-it\xi}dt.$$    
 In this paper    vectors in $\C^N$ are  to  be considered as column-vectors, however, to save space,   we shall write $x=(x_1,x_2, \dots, x_N)$ to denote the  column-vector whose components are $x_1, \dots,x_N$.  
  Let $t_o$ be a positive number; we shall say that  a subspace $H$ of $\Ld$ is $t_o$-shift-invariant if $H$ is invariant  under  translation by $t_o.$
 Given a subset  $\Phi=\{ \varphi_j,  j=1,2 \}$ of $ H,$     we denote by $ \Efi$ the set    $$ \Efi =\{\tau_{nt_o}\varphi_j  \quad  n\in\Z, \ j=1,2\},$$ where   $\tau_a f(x)=f(x+a).$   
 The family  $ \Efi$ is a   frame  for $H$  if  and  only if there exist two constants $0<A\le B$ such that $A\|f\|^2\leq \sum_{j=1}^{2}\sum_{n\in \Z}  |\langle f,  \tau_{n t_o}\varphi_j \rangle  |^2 \leq B \|f\|^2 $ for all $f\in  H;$ the constants $A$ and $B$  are called frame bounds. 
 Denote  by   $\Tfia: H \rightarrow \ldCN$ the    adjoint of $\Tfi,$ defined by $(\Tfia f)_j(n)= \langle f,  \tau_{n t_o}\varphi_j  \rangle, $ $ n\in \Z ,  j=1,2.$
The operator  $\Tfi\Tfia:H\rightarrow H$  is called {\it frame operator.} 
Denote  by  $\Phi^*$ the family  $\{\varphi_j^*, j=1,2\},$   where
\begin{equation}\label{duale}
 \varphi_j^*=(\Tfi\Tfia)^{-1}\varphi_j  \hskip1truecm     j=1,2.
 \end{equation}
If $\Efi$ is a frame for $H$ then    $\Efistar $ is also a frame,   called 
 the  {\it canonical dual frame}, and  
$\Tfi\Tfistara=\Tfistar\Tfia =I$; explicitly    \begin{eqnarray}\label{expansionzero}
f= \sum_{j=1}^{2} \sum_{n\in \Z}  \langle f,  \tau_{n t_o}\varphi_j^* \rangle \tau_{n t_o}\varphi_j= \sum_{j=1}^{2} \sum_{n\in \Z}  \langle f,  \tau_{n t_o}\varphi_j \rangle \tau_{n t_o}\varphi_j^*  \quad  \forall f\in H.
\end{eqnarray}  
  The elements of  $\Phi$ are called {\it generators} of the frame and the elements of $\Phi^*$ {\it canonical dual generators}. Given a frame   $\Efi$ of  a Hilbert space $H,$  the expansion of an element  in terms of 
 the generators is not unique; a frame         $\{\tau_{k t_o} {{\varphi}_1^{{\, }_{d}}} , \  \tau_{k t_o} {{\varphi}_2^{{\, }_{d}}}  ,k\in \Z\}$   such that   \begin{equation}\label{generaleadue}f= \sum_{j=1}^{2} \sum_{n\in \Z}  \langle f,  \tau_{n t_o}{{\varphi}_j^{{\, }_{d}}}\, \rangle \tau_{n t_o}\varphi_j\qquad  \forall f\in H \end{equation} is called  {\it  dual frame} for $\Efi$.
 We remind that     the  coefficients $\{ \langle f,  \tau_{n t_o}\varphi_j^* \rangle, n, j\in \Z\}$ with respect to the \emph{canonical} dual frame have the minimal  $\ell^2$ norm among all the sequences  that represent an element $f$ in terms of a given frame $\Efi$.    If the family   $ \Efi$ is a frame for $H$ and the operator $\Tfi$ is injective, then $ \Efi
 $ is called a  {\it  Riesz basis}. 
\par   Let $\Efi$ be  a    frame   of a  $t_o$-shift-invariant  Hilbert space $H$ with canonical dual generators $\Phi^*$. Denote by $P$ the orthogonal projection on a  $t_o$-shift-invariant subspace $V$ of $H$, then $\EPfi$ and  $ E_{P\Phi^*,t_o}$ are  families of dual frames for $V$, i.e.  
     \begin{eqnarray}\nonumber%\label{Projection}
 f= \sum_{j=1}^{2} \sum_{n\in \Z}  \langle f,  \tau_{n t_o}P\varphi_j^* \rangle \tau_{n t_o}P\varphi_j= \sum_{j=1}^{2} \sum_{n\in \Z}  \langle f,  \tau_{n t_o}P\varphi_j \rangle \tau_{n t_o}P\varphi_j^* \qquad \forall f\in V. \end{eqnarray}  
Note that   $ E_{P\Phi^*,t_o}$ is not necessarily the canonical dual.\\ 
   In this paper we study  families    $\Efi$, $\Phi=\{\varphi_1, \varphi_2\}$ where $\varphi_1$  and $\varphi_2$ have Fourier transform  $\widehat{ \varphi_1}(x)=\chi_{[-\omega,\omega]}(x) $ and $ \widehat{\varphi_2}(x)=ix\chi_{[-\omega,\omega]}(x) $
  and $t_o$ and $\omega$   are such that $\omega\le \frac{2\pi}{t_o} \le 2\omega$. The interest  of  these  families in applications resides in   their connection with  derivative sampling formulas  for the space $\BL$.   To simplify notation, throughout the paper we shall set  $$h=\frac{2\pi}{t_o}.$$ 
The family $E_{\Psi,t_o}$,  $\Psi=\{\psi_1, \psi_2\}$ defined by 
   \begin{equation}\label{fi1fi2BR} \widehat{ \psi_1}(x)=\chi_{[-h,h]}(x)\hskip 2truecm  \widehat{\psi_2}(x)=ix\chi_{[-h,h]}(x) 
 \end{equation}  is a Riesz basis  for the space $B_{h}$ 
   (see     \cite[p.135]{Hi}).  
   The Fourier transforms of  the dual generators are      
\begin{equation}\label{dualiRiesz1}  \what{\psi_1^*}(x)=
\frac{1}{h}(1-  \frac{|x|}{h})  \chi_{[-h,h]} (x)
\qquad  \what{\psi_2 ^*}(x)=
\frac{i}{h^2} \,{\rm sign}(x) \chi_{[-h,h]}(x),  \end{equation}
hence  the dual generators are
  \begin{equation}\nonumber %\label{riesz}
\psi_1^*(x)=\frac{1}{ \sqrt{2 \pi}} \snc^2(\frac{h x}{2}) \qquad \psi_2^*(x)=- \frac{1}{ \sqrt{2 \pi}}\,  x\,  \snc^2(\frac{h x}{2}),
\end{equation} 
where  $\snc(x) ={\sin(x)}/{x}$. For any $f\in B_h$  the coefficients of the  expansion    in     (\ref{expansionzero})  are
  \begin{equation}\nonumber%\label{coefficienti}
  \langle f,\, \tau_{-n t_o} \psi_1\rangle=\sqrt{2\pi} f( nt_o)\hskip1truecm \langle f, \, \tau_{-n t_o} \psi_2 \rangle=- \sqrt{2\pi} f'(nt_o)\quad  \forall n\in \Z.\end{equation}
thus       the   expansion formula   (\ref{expansionzero})  becomes
  \begin{equation}   \label{dsamp:riesz} f(x)=\sum_n \, \Bigl( f(nt_o)  \snc^2\frac{h}{2}\big(x-nt_o\big)\, +\, f'(kt_o) (x-nt_o)\,  \snc^2\frac{h}{2}\big(x-nt_o\big) \, \Bigr).
 \end{equation}
for  any $f\in B_h.$ Note that the convergence is uniform. This 
formula is called   a {\it first order derivative sampling formula}.   
The  sampling frequency     for an $N$-channel formula is $\omega/
N\pi$ and  is called {\it Nyquist frequency}.   By using   a frame  instead of 
a Riesz basis, in  \cite{SF} the authors obtained a two-channel   
derivative   formula where the    sampling frequency in each 
channel  is  greater  than  the Nyquist frequency  $\omega/2\pi$
({\it first order derivative oversampling formula}). The frame  is obtained  by  projecting on  $\BL$    the Riesz basis  $E_{\Psi,t_o}$  for $B_{h } $     where   $h=2\pi/t_o$ and 
  $\omega<h.$     We  denote by $\rf$   the ratio \begin{equation}\label{erreffe} \rf= \frac{\omega} {h}=\frac{\omega t_o}{2\pi}.\end{equation} 
Since the projection commutes with  the translations, to project  the family $E_{\Psi}$ and  the dual family, it is sufficient to project on $\BL$  the generators and the dual generators. Thus, by (\ref{fi1fi2BR}), we obtain that    the  generators  of the frame $\Efi$ are  defined by 
 \begin{equation}\label{fi1fi2} \widehat{ \varphi_1}(x)=\chi_{[-\omega,\omega]}(x)\hskip 2truecm  \widehat{\varphi_2}(x)=ix\chi_{[-\omega,\omega]}(x) 
 \end{equation}
and  by  (\ref{dualiRiesz1}) 
we obtain      the Fourier transform of  the dual  generators
    \begin{equation}\label{dualgen}  
   \what{\widetilde{\varphi}_1}(x)=
\frac{\rf}{\omega}(1-  \frac{\rf}{\omega}|x|)  \chi_{[-\omega,\omega]}(x) (x)
\qquad \what{\widetilde{\varphi}_2}(x)=i \frac{\rf^2}{\omega^2}
  \,{\rm sign}(x) \chi_{[-\omega,\omega]}(x).  
  \end{equation} 
A  simple calculation shows that  
  \begin{align}\nonumber
\widetilde{\varphi}_1(x)&= \frac{1}{\sqrt{2\pi}}\Big[2 \rf
\, 
(1-\rf
)\snc( \omega  x ) +\rf^2 
 \snc^2( \frac {\omega  x}{2} ) \Big] \\ { \ } \label{antif}\\
\widetilde{\varphi}_2(x)&=- \frac{1}{\sqrt{2\pi}}\ x  \rf^2
 \,  \snc^2( \frac {\omega  x}{2} ). \nonumber
 \end{align} 
Thus we obtain the  following derivative oversampling  
  formula  
\begin{equation}\label{proiettata1}
f(x)=\sqrt{2\pi}\sum_{k\in Z}\Bigl( f(k t_o)\tau_{kt_o} \widetilde{\varphi}_1   -  f'(k t_o)
  \tau_{kt_o} \widetilde{\varphi}_2  \Bigr) \quad \forall f\in B_{\omega} \end{equation}
  \cite{SF}.   We observe  that $1/\rf={2\pi}/{\omega t_o}$ is the  ratio between the sampling frequency 
  and the  Nyquist frequency and that   
$\rf \in (0,1).$ We shall  be mainly interested   to the case  $1/2< \rf<1$, since  for $0<\rf\le1/2$        it is  possible to use    one channel  separately for  the function and for its derivative  \cite{SF}.
   Note that if $\rf$ is close to 1, the  frame  is close to a Riesz basis, while if $\rf$ is  small the frame is very redundant.
    %\graphicspath{{FILESFIG}}
 \section{The system  for the recovery of missing samples} \label{s:Procedure}   
In this section  we briefly describe  the  method     for the recovery of  a finite number of missing samples via     two-channel   derivative oversampling.  We may rewrite     equation  (\ref{proiettata1}) as
 \begin{equation}\label{proiettata3}
f(x)=\sqrt{2\pi}\sum_{i=1}^2 \sum_{k\in Z} (-1)^{i-1}  f^{(i-1)}(k t_o) {\widetilde{\varphi_i}}(x-k t_o)  \end{equation}
and   computing the derivative of both sides,  we obtain  \begin{equation}\label{proiettata2}
f'(x)=\sqrt{2\pi}\sum_{i=1}^2 \sum_{k\in Z}(-1)^{i-1} f^{(i-1)}(k t_o) {\widetilde{\varphi_i}}'(x-k t_o).  \end{equation}
  Here 
    \begin{align}
  \widetilde{\varphi}_1^{\prime}(x) &= \frac{2\rf}{\sqrt{2\pi  } }   \frac{1}{x} \Big[  (1-\rf)
   \big(\cos( \omega x)  -\snc(\omega   x)\big ) + \label{derdualiferr} \\
  &   \hskip3truecm  + \rf\, \snc(\frac{\omega}{2} x)  \big(\cos(\frac{\omega x}{2}    )-\snc(\frac{\omega x}{2} )\big) \Big] \nonumber 
\\  
  \widetilde{\varphi}_2^{ \prime }(x) &=-   \frac{1}{\sqrt{2\pi}} \rf^2\, \snc(\frac{\omega x}{2} )\Big[ 2\cos( \frac{\omega x}{2} )- \snc( \frac{\omega x}{2} ))\Big]. \label{derfi12}   \end{align}
Observe that $  \widetilde{\varphi}_2^{ \prime }(0)=0$.\par\noindent
Let  $\cU=\{ l_1,l_2 \ldots,l_N\} \subset \Z$ and   let   $\{f(l_jt_o),\  f^{\prime}(l_jt_o) \ 1\le j\le N  \}$ 
be the corresponding     set of  missing samples. By evaluating  equations (\ref{proiettata3}) and (\ref{proiettata2})   in $\ell_k t_o$ and   separating the unknown samples from the known ones, we obtain
\begin{equation}   \begin{aligned}
 f(l_k t_ o) -\sqrt{2\pi}\sum_{i=1}^2\sum_{j=1 }^N (-1)^{i-1} f^{(i-1)}(l_jt_ o)\ \widetilde{\varphi_i}((  l_k -l_j)t_ o )&=c_k  \\
f'(l_k t_ o) -  \sqrt{2\pi}\sum_{i=1}^2
\sum_{j=1 }^N (-1)^{i-1} f^{(i-1)}(l_jt_ o)\widetilde{\varphi_i}{}^{\prime} ((  l_k -l_j)t_ o )\ 
&=c_{k +N},   
\end{aligned} \label{samp:4} \end{equation}
$1\le k\le N,$ where 
 \begin{equation} \begin{aligned} c_k&= 
 \sqrt{2\pi}\sum_{i=1}^2 \sum_{n\in \Z\setminus \cU}
 { {(-1)^{i-1}} } f^{(i-1)}(n t_o) {\tilde{\varphi_i}}( {l_k t_ o-nt_ o})
\\  c_{k+N}&=\sqrt{2\pi}\sum_{i=1}^2  
\sum_{n\in \Z\setminus \cU} { {(-1)^{i-1}} }f^{(i-1)}(n t_o) {\widetilde{\varphi_i}}
{}^{\prime} ( {l_k t_ o-nt_ o}).
 \end{aligned}\label{ckck}\end{equation}
Equations (\ref{samp:4}) form    a system of $2N$ equations in   the  $2N$ unknowns    
 \begin{equation}\label{XF}
  f(l_1 t_o), 
\ldots,f(l_N t_o), f^{\prime}(l_1 t_o),
 \ldots,f^{\prime}(l_N t_o) \end{equation}
which  can be written in matrix form:
 denote by    $\Sfi=\Sfi(\cU,\rf)$    the real matrix
\begin{equation}\label{Sfi} \Sfi=
\begin{bmatrix} S_{1 1}&  {\ }S_{1 2}{\ } \\  
  {\ }S_{2 1}{\ }  & {\ }S_{2 2}{\ }   \\  
   \end{bmatrix} 
\end{equation} 
where   $S_{11},S_{12},S_{21},S_{22}$ are the submatrices  whose entries are
\begin{align}
S_{11}(k,j)&= \sqrt{2\pi} \   \tilde{ \varphi_1}(( {l_k  -l_j)t_ o})
 \hskip.7truecm 
S_{1 2}(k,j)=- \sqrt{2\pi} \   \tilde{ \varphi_2} (( {l_k  -l_j)t_ o})\nonumber\\
{ \ }  \label{matrice}\\
S_{2 1}(k,j)&= \sqrt{2\pi}\  \tilde{ \varphi_1}^{\prime}(( {l_k  -l_j)t_ o})
 \hskip.7truecm 
S_{2 2}(k,j)=  - \sqrt{2\pi} \ \tilde{ \varphi_2}^{\prime}(( {l_k  -l_j)t_ o})
\nonumber
 \end{align}
$ k,j=1,\dots,N.$ Then the system (\ref{samp:4}) may be written \begin{equation} \label{sistema}(I -\Sfi)Z=\Cfi
\end{equation}
where  $I$ is  the $2N\times 2N$ identity  matrix,
    $\Cfi  =(c_k)_{k=1}^{2N}$   
is  defined by (\ref{ckck}) and $Z=\big(X_1,X_2,\ldots X_{N},Y_1,Y_2,\ldots Y_{N}\big)$. The unknowns  $X_k$ are the missing samples of the function and the $Y_k$ are the missing samples of   the derivative. \\  
Notice that, by (\ref{antif}), (\ref{derdualiferr}) and (\ref{derfi12}), if $\rf=1$ then  $\Sfi$ is the identity matrix. The four submatrices are real,  
 $S_{11}$ and $S_{22}$ are  symmetric,  while  $S_{1 2}$ and   $S_{ 2 1}$ are  antisymmetric.  Moreover, if the distance between two consecutive $l_k$ is constant, then the four submatrices  are Toeplitz, but $\Sfi $ is not Toeplitz. \\
Suppose that  in (\ref{XF})     only the values   
$\{X_i= f(l_i t_o), \quad i=1,\dots,N\} $
   are missing, while all the values of  $f'$ are known. By considering only the first $N$ equations  of the system (\ref{sistema})   and by separating the  known  from the  unknown   samples,  we obtain        
       \begin{equation}\nonumber%\label{solof}
(I -S_{11})X= C_1+S_{12}Y, 
\end{equation}
 where $I $ is  the $ N\times  N$ identity matrix, $X=\bigl(X_k\bigr)_{k=1}^{N}$, 
  $Y=(f^{\prime}(l_k t_o) \big)_{k=1}^{N} $
and    $C_1=(c_k)_{k=1}^{N}$. This equation   may  be rewritten 
  \begin{equation}\label{solof}
(I -S_{11})X  =B_1    
\end{equation}
where $B_1=C_1+S_{12}Y$.   Similarly, if   in (\ref{XF})     only the values   
$\{Y_i= f'(l_i t_o), \, i=1,\dots,N\} $ are missing, and the samples of $f$ are known, one solves the system  \begin{equation}\label{solofp}
  (I -S_{22})Y =B_2 
\end{equation}
in the unknowns $Y_i,i=1,\dots,N$, where $C_2=(c_k)_{k=1+N}^{2N}$
 and $B_2=C_2+S_{21}X.$\\
If $\rf\in (0,1)$,   then
 the   eigenvalues of the symmetric matrices $S_{11}$    and 
$S_{22}$ are in the interval $(0,1)$ \cite{SF}.   
This  implies that the recovery of a 
finite number of missing samples of the function is possible,   if  all 
the samples   of  the derivative are known  and, vice-versa,  
samples of  the  derivative can be recovered, if  all the samples    of 
the function  are known.  In the next section we shall find bounds for the minimum and maximum eigenvalues of  the matrices $S_{11}$   and $S_{22}$. \\
To the best of our knowledge no results are known about the 
possibility of solving  system (\ref{sistema}), i.e. of recovering 
simultaneous missing samples of the function and of its derivative. 
Notice that the system is solvable if and only if $1$  is not an 
eigenvalue of the matrix $\Sfi$. On the basis of numerical evidence  we  conjecture that   all its eigenvalues are real and that they lie in 
the interval $(0,1)$ for all $\rf\in (0,1)$. In Section 4 we present several numerical 
experiments supporting our conjecture.
 
  \section{The stability of the  matrices.}% \label{s:stability}
  This section is dedicated to the study of  the eigenvalues of the matrices $\Sfi$, $S_{11}$  and $S_{22} $ in the two-channel system (\ref{sistema}). 
First, we briefly summarize the one-channel case 
and  recall   some stability results  matrix obtained in \cite{F}. 
The   one-channel formula  is    \begin{equation}\label{rico2}f(x)=
\frac{ \omega t_o}{ { \pi}}  \sum_{n\in \Z}f(nt_o) \frac{\sin\omega(x-nt_o)}{\omega(x-nt_o)} \qquad \forall f\in \BL
   \end{equation} where  $t_o< \frac{ \pi}{\omega} $. 
In this case  the oversampling parameter is  $\rf= { \omega t_o}/  { \pi}    $. Let  $U=\{l_1,l_1,\dots,l_{N}\}$  be a set of    integer numbers and 
\begin{equation}\label{punti mancano}
\{ f(l_kt_o),k=1,\dots,N\}
\end{equation}   the set  of missing samples; then the system to be solved is 
  \begin{equation}\label{sistema:0}(I-\S)X=B\end{equation} where  $I$ is the identity $N\times N$ matrix  \begin{equation}\label{S1can}
 \S({j,k}) = \rf\,  \snc( \pi \rf ( l_j - l_k)),  \quad \quad j,k =1,\dots,N,
 \end{equation}   $X=\{ f(l_kt_o),k=1,\dots,N\}$, $B=(b_1,b_2,\dots,b_{N})$, and  
 \begin{equation}\label{B}b_j=\rf \sum_{k\notin U}  f(kt_o)\snc(\pi \rf (l_j-l_k)) \quad \quad j=1,\dots,N.\end{equation} 
The matrix $\S$ is  symmetric and positive definite. In  \cite{F} the author shows that    all its eigenvalues are in $(0,1) $  and  observes that there are two  situations  that lead to a ill-conditioned problem: when  $\rf$ is close to 1 and when the integers $l_k$ are   contiguous.   In the first  case, the   frame is close to a Riesz basis, when the recovery is impossible and   $\S$ becomes   the identity matrix.
In the second  case  the maximum eigenvalue of the matrix        
 $  \S $ grows rapidly with $N$; moreover it can be  close to    1 also for small    values of  $N$ (see Figure~5 in \cite{F}).  In both cases the  spectral condition number of the  matrix $I  -\S$, which controls the propagation of errors on the data, becomes very large.\par\noindent
Next, we investigate   the eigenvalues of the matrix  $\Sfi$ in the two channel case (see (\ref{Sfi}))  and  of its submatrices $S_{11} $ and $S_{22}$.  Since the latter are  real symmetric matrices, their eigenvalues give their  condition number.
We shall denote by  $\lfloor a\rfloor$  the maximum integer less than or equal to $a$.  
   Given a $N\times N$ matrix, we shall denote with $\lambda_j(A), j=1,\dots,N$ its eigenvalues and by $\lambda_{min}(A) $ and $\lambda_{max}(A) $ its  minimum and maximum eigenvalue.  \\
  Let  $\cU=\{ l_1,l_2 \ldots,l_N\} \subset \Z$ and   let   $\{f(l_jt_o),\  f^{\prime}(l_jt_o) \ 1\le j\le N \}$ 
  a set of  missing samples. Then 
   \begin{align}\nonumber
   \lambda_{min}(S_{11})&<2\rf-\rf^2 <\lambda_{max}(S_{11}) \\
     \lambda_{min}(S_{22})&< \rf^2 <\lambda_{max}(S_{22}).\nonumber\end{align}   
 where $\rf$ is the oversampling parameter (\ref{erreffe}).
Indeed, by (\ref{matrice}) and (\ref{antif}), the trace of $S_{11}$ is $N(2\rf-\rf^2)$; by (\ref{matrice}) and (\ref{derfi12}) the trace of $S_{22}$ is $N\rf^2$. 
Moreover, since the  entries of the matrix   $\Sfi $ are real, its  eigenvalues are complex conjugates, thus its trace is $N(2\rf-\rf^2 )+N\rf^2=\sum_{i=1}^{2N}   Re( \lambda_i (\Sfi))$; hence 
   \begin{equation}\label{rsepara:1}
  Re( \lambda_{min}(\Sfi))<\rf <Re(\lambda_{max}(\Sfi)). \end{equation}   
\noindent We observe that, if the oversampling parameter $\rf$ tends to 1, i.e. the frame tends to a Riesz basis, then  $\Sfi$ tends to the identity matrix.\par
 Following Ferreira,     we shall now   consider the  case in  which 
the  set    that locates the  positions 
 of  the missing samples  is 
  $\cU= \{ m \,i_1,m\,i_2,\dots,m\, i_N\} ,$    where  $m$ is an integer   and  
$\cI=\{ \,i_1,\,i_2,\dots\, i_N\} $ is a set of  integers;    in the  following  we shall denote  such sets        
by $m\cI$.   The interest  for studying these cases lies in the   
technique  of  interleaving    the samples of a  
signal, prior to their transmission or archival;  the advantage of this 
procedure is that   the transmitted (or stored)  information becomes 
less sensitive to the burst errors that typically affect  contiguous  set 
of samples (\cite{F}). We shall investigate how the stability of the method   depends on the  interleaving factor $m$. 
  In   Proposition \ref{mrinteger} and  
   Theorem  \ref{eigS11} 
    below    we  find  estimates for the eigenvalues of the matrices     $S_{11}, S_{22},$
    thus    generalizing   a result of Ferreira  for one channel    \cite[Theorem~1]{F}. \par\noindent
First we   consider the case  $ m r$ integer.  The following proposition, which  extends  to two channels   a result in \cite{F}, shows 
that  in this case   the matrix $\Sfi$ is lower  triangular and that $N$ eigenvalues are equal to $2\rf-\rf^2$ and $N$ are equal to  $\rf^2.$ Hence   if $\rf$ is   rational $\rf=p/q,$ 
one can take $m$ equal to a multiple of $q$   and obtain a  lower-triangular 
matrix $\Sfi$.     
 \begin {proposition}\label{mrinteger}
 Let $m$ be a positive integer; suppose   $\cU=\{m\,i_1,m\,i_2,\cdots,m\, i_N\}$, where $i_j,j=1,\dots,N$  are  integers and let $\rf$ be   a real number in  $(0,1).$ If   $ m\rf $ is an integer,  then     $$S_{11}= (2\rf-\rf^2)I \qquad  S_{22}=\rf^2 \, I   \qquad S_{12}=0.$$
  Moreover   the entries of the matrix $S_{21}$  are 
$$ 
  S_{21}(k,j) =  \begin{cases} 
(1-\rf)  {\omega}/{(m \pi (i_k-i_j))} 
  &\mbox{if} \quad {k\not=j}  \\   
   0 &\mbox {if}\quad  {k=j} .
  \end{cases}
  $$ 
 \end{proposition}
 \begin{proof}
Since $t_o=2\pi\rf/\omega$, by  (\ref{matrice})   with $\ell_k=mi_k$, one obtains that 
\begin{equation}\nonumber
S_{11}(k.j)=\sqrt{2\pi} \widetilde{\varphi}_1\bigl(2 \pi m\frac{\rf}  {\omega}(  i_k  -i_j) \bigr);
\end{equation}  
hence  by  (\ref{antif})  since  $m\rf$ is an integer 
\begin{equation}\nonumber
S_{11}(k,j) 
  = (2\rf \, (1-\rf) + \rf ^2)
\delta_{j,k}\qquad 1\le k\le N, \quad 1\le j\le N 
\end{equation}
   where $\delta_{j,k}$ is the Kronecker symbol. This shows that   $S_{11}=(2 \rf-\rf^2)I_N.$
    Similarly, by using (\ref{matrice})  and (\ref{derfi12}), we obtain that
$S_{22}=\rf^2 I $  and  $S_{12}=0 $. From (\ref{derdualiferr}) 
 one finds that  the off-diagonal entries of the matrix $S_{21}$ are $$S_{21}(k,j)= (1-\rf) \frac{\omega}{m \pi (i_k-i_j)} \quad k\not=j.$$
The diagonal entries in  $S_{21}$ are  equal to zero, since $ {\widetilde{\varphi}_1 }^{\prime}(0)=0.$   \end{proof}
In the  following   theorem  we find  estimates for  the minimum and maximum eigenvalues  of the matrix $S_{11}$ and $S_{22}.$   The estimates  do not depend on the number $N$ of missing samples and  are simple to compute.
 \begin {theorem}\label{eigS11} Suppose  that  {   $0<\rf <1$ } and $\cU=m\cI$,  where $\cI=\{\,i_1,\,i_2,\cdots,\, i_N\}$ and     $m$ is  a positive integer.  If $ m \,\rf$   is not   integer,   then for $j=1,\dots,N$
 \begin{align} 
 \minboundu &< \lambda_j(S_{11})< \maxboundu   
 \label{S11eig}\\
\minboundd &< \lambda_j(S_{22})< \maxboundd
 \label{S22eig} 
 \end{align}
 where 
  \begin{align*} 
 \minboundu&=\frac {D}{m}(1-\frac{D}{4m}-\frac{1}{4m }) ,  
\qquad \maxboundu=\minboundu+\frac{1}{m} \\
  \minboundd &=   \frac{D }{4m^2}(D-1) , \qquad  \qquad \   \ 
 \maxboundd = \frac {D }{4m^2 }(D+1)+\frac{\rf} {m}  \end{align*}
 and $D=\lfloor 2m \rf \rfloor$.
 \end{theorem}
 \begin{proof}  
From  (\ref{dualgen}), by Fourier inversion and a  change of  variable, we obtain
\begin{equation}\label{fi1tilde}
\widetilde{\varphi}_1(x)=\frac{1}{\sqrt{2\pi}}\int_{-\rf}^{\rf} (1-|y|) e^{i x y  {\omega}/{\rf}  } dy.
\end{equation}
Hence by (\ref{matrice})  with $l_k=m i_k$,
$$S_{11}(k,j)=
\int_{-\rf}^{\rf} (1-|y|) e^{2\pi i m y(  i_k  -i_j)} dy.$$
Let $x=(x_o,x_1,\dots,x_{N-1})$ be a column-vector in $  \C^N$  and denote  by $x^*$ its conjugate transpose.  Then \begin{align}\nonumber
x^*S_{11} x=\sum_{k,j=0}^{N-1} \overline{x}_k   S_{11}(k,j)\, x_j&=\int_{-\rf}^{\rf} (1-|y|) \sum_{k,j=0}^{N-1} \overline{x}_k  x_j  e^{2\pi i m y(  i_k  -i_j)} dy\\ \nonumber
&=   \int_{-\rf}^{\rf} (1-|y|)|\PP(y)|^2 \ dy ,\end{align}
 where  $\PP(y)$ is the $1/m$-periodic function
\begin{equation}\label{poly} \PP(y)=\sum_{k=0}^{N-1}x_k  e^{-2\pi i m y  i_k  } .\end{equation}
 By splitting the integral in two parts and changing variable   the left hand side of the above equation may be written
 \begin{equation}\label{xSx}x^*S_{11} x=\int_0^{\rf}(1-y)  M(y)\ dy\end{equation}
 where we have set \begin{equation}\label{M}M(y)=|\PP(y)|^2 +|\PP(-y)|^2 .\end{equation}
Next  we write the integral  in (\ref{xSx}) as a sum of $D+1$ integrals   \begin{equation}\label{xSx:1}
 x^*S_{11} x  =\sum_{k=0}^{D-1}\int_{k/2m}^{(k+1)/2m}(1-y)  M(y)\ dy+\int_{D/(2m)}^{\rf}(1-y)  M(y)\ dy. \end{equation} 
 Since  $M\ge 0$ we majorize  $1-y$ in each integral and  obtain the estimate 
 \begin{equation}\label{ineq:1}
 x^*S_{11} x   \le    \sum_{k=0}^{D-1}(1-\frac{k}{2m})\ \int_{k/2m}^{(k+1)/2m}   M(y)\ dy+(1-\frac{D}{2m})\int_{D/(2m)}^{\rf}   M(y)\ dy.
 \end{equation}
Next we prove that for all $k=0,\dots,D-1$
 \begin{equation}\label{eq:1}
  \int_{k/2m}^{(k+1)/2m}   M(y)\ dy=\frac{1}{m}. \end{equation}
  Indeed by (\ref{M}), by changing variables  and using the $(1/m)$-periodicity of $\PP,$
\begin{align}\nonumber  \int_{k/2m}^{(k+1)/2m}   M(y)\ dy&= \int_{k/2m}^{(k+1)/2m}   |\PP(y)|^2 \ dy +\int_{-(k+1)/2m}^{-k/2m}  |\PP(y)|^2 \ dy\nonumber \\
&=\int_{k/2m}^{(k+1)/2m}   |\PP(y)|^2 \ dy +\int_{(k+1)/2m}^{( k +2) /2m}  |\PP(y)|^2 \ dy\nonumber \\
&=\int_{k/2m}^{( k +2) /2m}   |\PP(y)|^2 \ dy=\int_{0}^{  1/m}   |\PP(y)|^2 \ dy.\nonumber 
 \end{align}
By (\ref{poly}) and the Plancherel formula, since  $x$ has norm equal to 1, we have 
 \begin{equation}\label{normfi} \int_0^{1/m} |\PP(y)|^2 \ dy=\frac{1}{m}.\end{equation} This concludes the proof of equation  (\ref{eq:1}).  Next we prove that 
 \begin{equation}\label{eq:2}
  \int_{D/(2m)}^{\rf}   M(y)\ dy<\frac{1}{m}.\end{equation}
Indeed, from  (\ref{M}), by  changing variable in  the integral and using the $1/m  $-periodicity of $\PP$,   we obtain
\begin{align*} \int_{D/(2m)}^{\rf}   M(y)\ dy&= \int_{D/(2m)}^{\rf}   |\PP(y)|^2\ dy+ \int_{D/(2m)}^{\rf}   |\PP(-y)|^2\ dy\\ &=\int_{D/(2m)}^{\rf}   |\PP(y)|^2\ dy+ \int_{D/m-\rf}^{D/(2m)}   |\PP(y)|^2\ dy\\& = \int_{D/m-\rf}^{\rf}  |\PP(y)|^2 \ dy.
 \end{align*}
The length $2\rf-D/m$ of the interval of integration   is less than  the period  $1/m$ of $\PP,$   thus   inequality (\ref{eq:2})  follows  from (\ref{normfi}). \\
From (\ref{ineq:1}), by using (\ref{eq:1}) and (\ref{eq:2}),  we obtain
 \begin{equation}\label{ineq:2}
 x^*S_{11} x   \le  \frac{1}{m}  \Big[\sum_{k=0}^{D-1}(1-\frac{k}{2m})\  +(1-\frac{D}{2m}) \Big]= \frac{D}{m}\big(1  -\frac{D}{4m}-\frac{1}{4m }\big)+\frac{1}{m} 
 \end{equation}
Thus we have proved  the upper bound for the eigenvalues of  $S_{11}.$
From  (\ref{xSx:1}), by  observing that the second integral is  positive    and  using  (\ref{eq:1})    we obtain
$$ x^*S_{11} x  \ge \sum_{k=0}^{D-1}\int_{k/2m}^{(k+1)/2m}(1-y)  M(y)\ dy >  \frac{1}{m}\sum_{k=0}^{D-1} \big(1-\frac{k+1}{2m}\big)=  \frac{D}{m}\big(1-\frac{D}{4m}-\frac{1}{4m}\big).$$
This proves  the lower bound  and concludes  the proof of inequality (\ref{S11eig}). The omit the  proof of inequality  (\ref{S22eig}) which is similar.
\end{proof}
  \begin{table}
\begin{tabular}{|l|c|c|c|c|c}\hline
 $\rf$&$\minboundu$ &\, $\lambda_{min}(S_{11})$&$\lambda_{max}(S_{11})$&  $\maxboundu $         \\  \hline	
 0.55 &0.719 &0.768 &0.811   &  0.844  \\
 .6& 0.773 &	0.813 	&0.859 	&0.898 \\ 
 .7& 0.859 	&0.903  	&0.926 	&0.984   \\
 .8 &0.891	&0.946	&0.967 	&1.016 \\
 .9 &  0.929 	&0.984  &	0.998 &	1.055  \\ 
.95& 0.936	&0.996	&0.999	&1.063  \\  
	 \hline  \end{tabular}
  \caption{$  \lambda_{min}  (S_{11}), \, \lambda_{max} (S_{11})$ and their estimates for $\cU=\{0,8,16,24\}$, $m=8$  and several  values of $\rf$.}
  \label{tavolaS11}
\end{table}
 \begin{table}\begin{tabular}{|l|c|c|c|c|c}\hline
 $\rf$&$\minboundd $&\, $\lambda_{min}(S_{22})$&$\lambda_{max}(S_{22})$&$\maxboundd$      \\  \hline	
.55&	0.219 &	0.271 &	0.315 &	0.350 \\
.6&	0.281&	0.317& 	0.391 &	0.427 \\
.7 	&0.430 &	0.470&	{0.535}&{0.603}\\
.8 	& 0.516 &	0.594&{0.659}&	{0.709}\\
.9&	0.711&	0.766&	0.871&	0.932\\
.95	&0.820 &	0.877&	0.962 	&1.056\\
	 \hline  \end{tabular}
  \caption{$  \lambda_{min}  (S_{22}), \, \lambda_{max} (S_{22})$ and their estimates  for  $\cU=\{0,8,16,24\}$    and several  values of $\rf$.}
\label{tavolaS22}\end{table}
 In   Table~\ref{tavolaS11}   and Table~\ref{tavolaS22} we   compare the        minimum and maximum eigenvalues  of  the matrices  $S_{11}$ and  $S_{22}$  with their  estimates  given  by Theorem  \ref{eigS11}  
  for  $\cU=\{0,8,16,24\}$ $m=8$  and various values of $\rf.$  
From  Proposition \ref{mrinteger} one  can see   that, if $  m\rf$ is integer,  then the  eigenvalues of $S_{22}$ are  smaller than the eigenvalues of $S_{11}$.   
The following corollary  shows   that this is also  true  when 
 $m\rf$ is not an integer, provided that  $m$ is sufficiently large.  
 \begin {corollary}
 Let $m$ be a positive integer, $\cU=m\cI$ and $\cI=\{\,i_1,\,i_2,\cdots,\, i_N\}$,   $0<\rf<1.$       If $m \rf$  is not an integer and $m> (1+2\rf)/(2\rf (1-\rf)),$ then  \begin{equation}\label{max}
   \lambda_{max}(S_{22})<\lambda_{min}(S_{11}).\end{equation}   
 \end{corollary}  
  \begin {proof} By (\ref{S11eig}) and  (\ref{S22eig})  it is sufficient  to  prove that      if   $m> (1+2\rf)/(2\rf (1-\rf)),$   then $ \maxboundd< \minboundu.
$  This  inequality is equivalent to
 $  {D^2}/{2m} +{D }/{2m}+\rf-D<0.$ 
 Since  $2mr-1<D\le2mr,$ we have
 $$  \frac{D^2}{2m} +\frac{D }{2m}+\rf-D<2m\rf(\rf-1)+2\rf+1 $$ 
from which   the corollary follows.
  \end{proof}
Numerical experiments    show that for   $m<   (1+2\rf)/(2\rf (1-\rf) $ the  maximum eigenvalue of $S_{22}$ may be  larger that  the minimum eigenvalue  of $S_{11}$.\par
We  shall now describe   the behavior of the eigenvalues of the matrix $\Sfi$ in dependence of the parameters  $\rf $ and $N$. The numerous    experiments that we have performed  suggest the conjecture that  the  eigenvalues of this matrix are real,   positive, and  less than 1.   In  what follows   we have ordered the  eigenvalues  so that $\lambda_i<\lambda_{i+1}$ for $i=0,\dots,N-2$. 
 \begin{figure}[h]
\begin{center}
\centering\setlength{\captionmargin}{0pt}%
 \includegraphics[width=8.5cm,height=7cm]{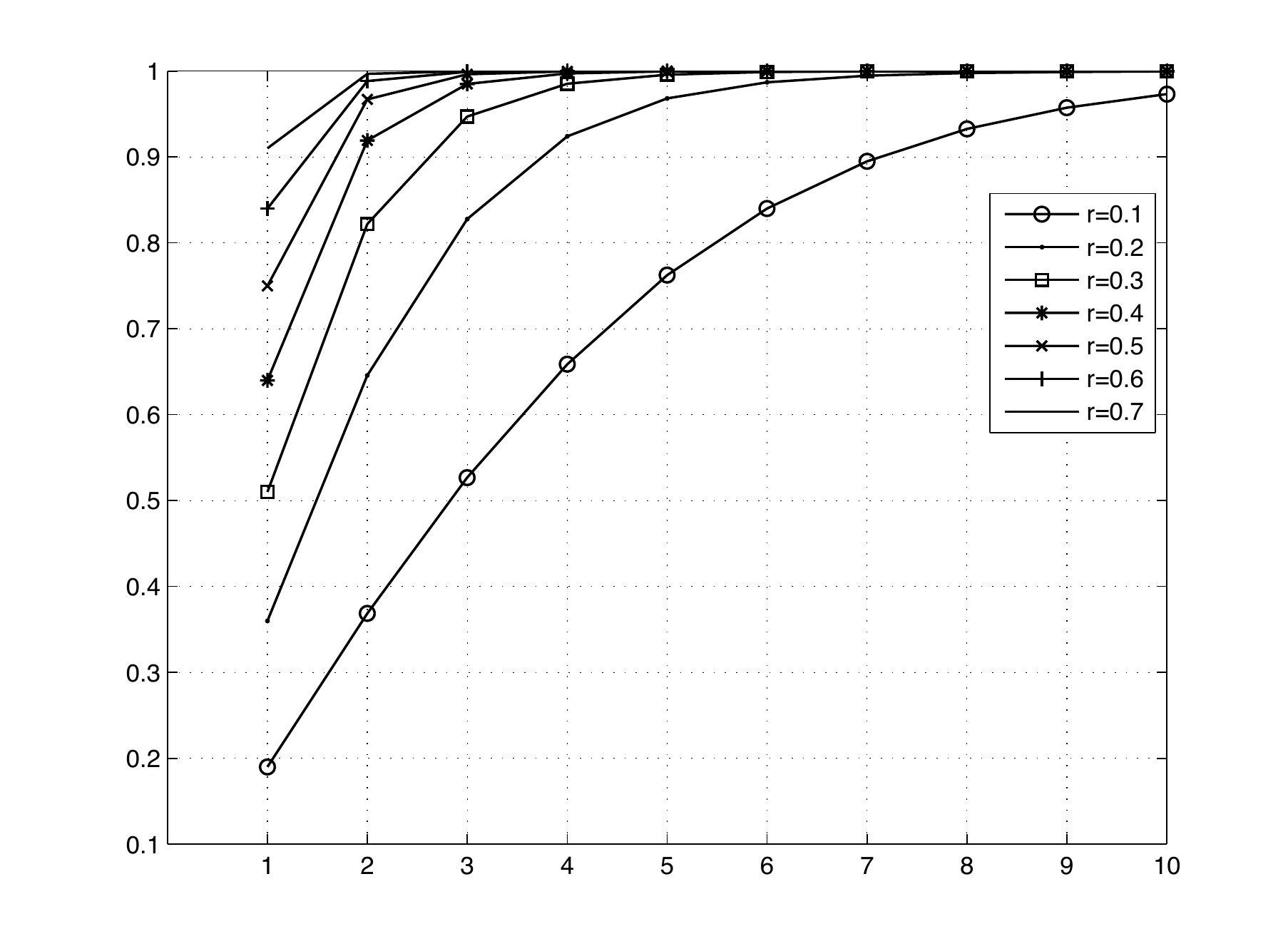}
 \end{center}
  \vskip-.7truecm
 \caption{ \small{Largest eigenvalues of $\Sfi$ as   function of $N$  for   $\cI=\{0,1,\cdots,N-1\}$, for  various values of $\rf$. }}
\label{versusN}
\end{figure}
 Figure  ~\ref{versusN}    shows the  largest  eigenvalue of $\Sfi$ as 
function of the  number $N$ of contiguous points,      for various 
values of $\rf$.  Note that,   as $\rf$ approaches 1,  the largest  eigenvalue 
gets close to 1, i.e.  the smallest eigenvalue of    the matrix $I-\Sfi$ 
tends to zero.   This happens   even for a small number of missing samples, as for   the one-channel case (see  
\cite{F}, Figure 5). Thus, since the spectral condition number of a matrix  is  greater 
or equal  to the ratio of the absolute values of the largest and the 
smallest eigenvalues,     $\textrm{cond}(I-\Sfi)$  becomes very 
large.  
 %FATTO COL PROGRAMMA  maxautovSSversusNm2nera.m  
\noindent In Table~\ref{condIMS}  we show the spectral condition numbers of  $I-\Sfi $ for   $\cU=\{0,1,2, \cdots, 9\}$,   for various values of $\rf$.   
  \begin{table} \begin{tabular}{|c|c|c|c | }\hline
 \small{$ \rf $}&\small{$\textrm{cond}(I-S) $} &\small{$\rf$}   &
\small{$\textrm{cond}(I-S)   $ }
    \\  \hline
.1&	8.571\,e+01   &.6 	&2.829\,e+13	\\
 .2&	5.870\,e+04	&.7 	& 1.661\,e+16 	\\	
 .3&	6.187\,e+05 &	.8&	2.474\,e+17	\\
.4&1.133\,e+08 &.9 	&4.096\,e+17	 \\
.5&3.513\,e+10 &1	&	  \\
	 \hline  \end{tabular}
  \caption{ Condition number of  $I-\Sfi$    for  $\cU=\{0,1,2,\dots, 9\}$     and   $\rf=0.1,0.2,\dots,1$.}
\label{condIMS}\end{table} 
\begin{figure}[h]
\begin{center}
\begin{tabular}{ l r} 
\hskip-1truecm
\includegraphics[width=7cm]{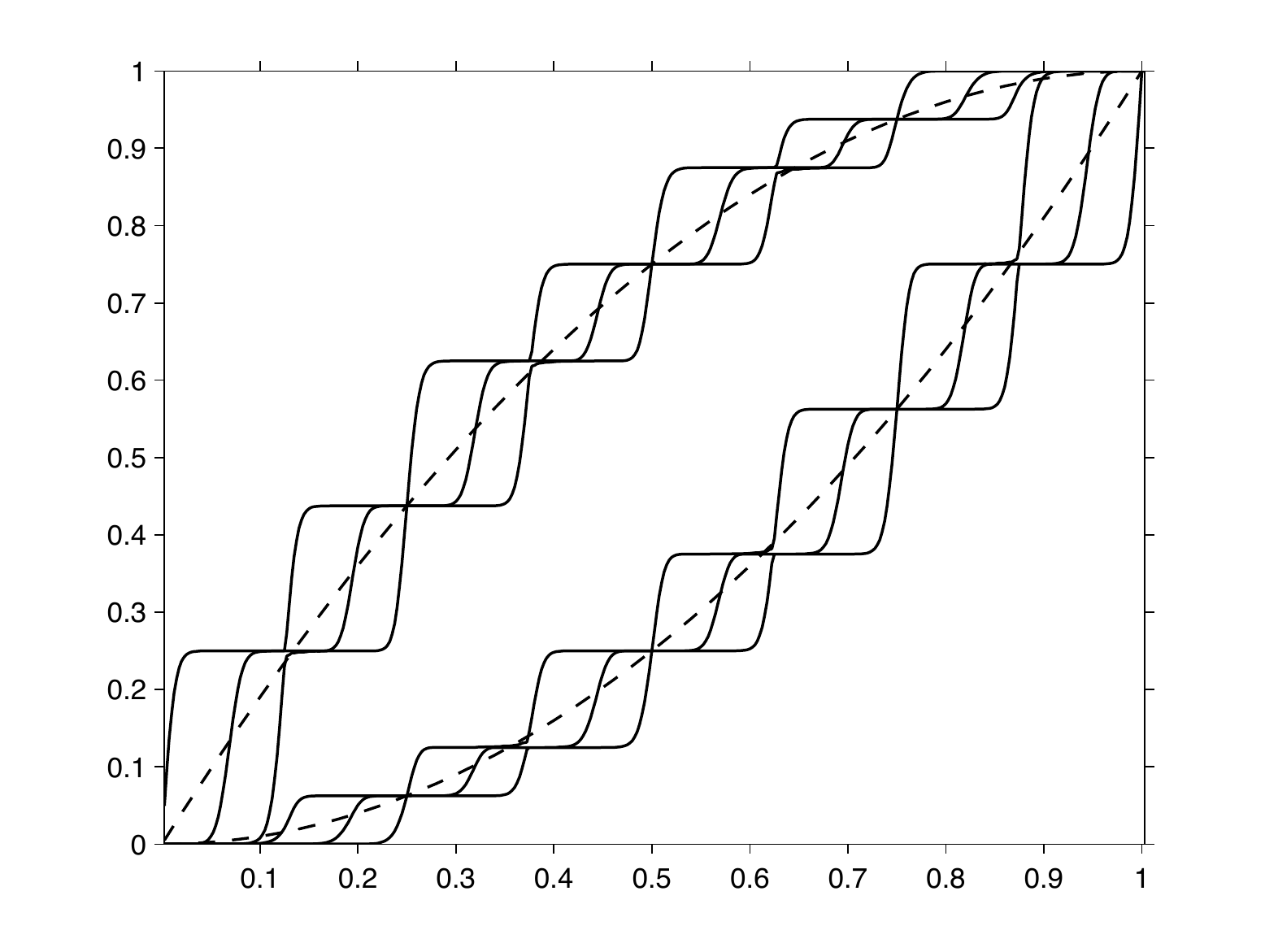}
\includegraphics[width=7cm]{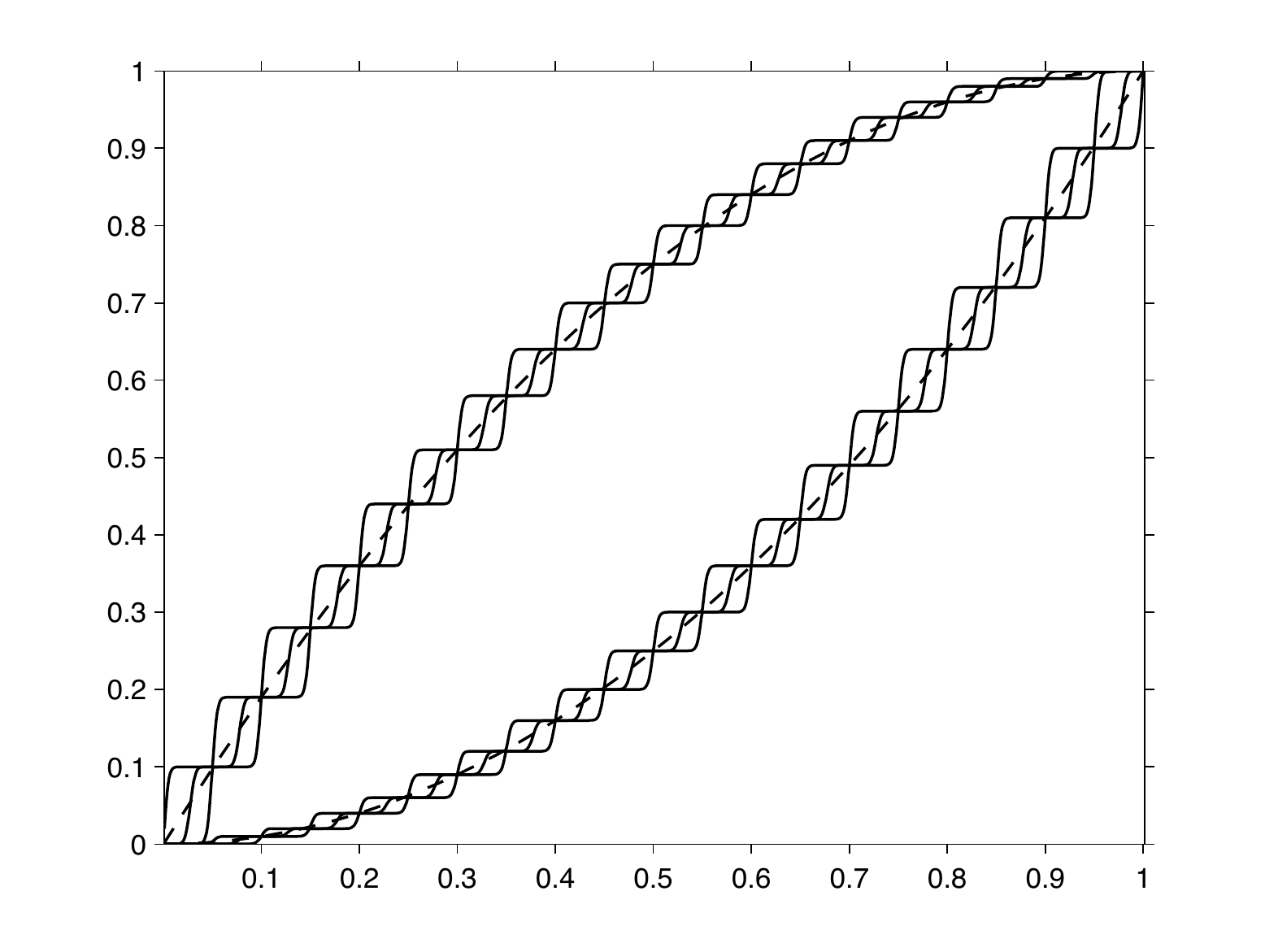} \end{tabular}
\end{center} 
 \vskip-.5truecm
\caption{ \small{Six eigenvalues  of $\Sfi$
  as   functions of $\rf$ for $\cU=\{0,1,2,3,\dots,9\}$, $m=4$ (left) and $m=10$ (right). Dashed lines: parabolas $r^2$ and $2r-r^2$}}
\label{versusr}
\end{figure}
\par \noindent Figure~\ref{versusr}     shows  some of      the eigenvalues of $\Sfi
$,     namely $\lambda_{i},i=1,5,10,11,15,20$,   as functions of  the 
parameter $r$   for $\cI=\{0,1,2,3,\cdots,9\}$,   $m=6$ (left) and $m=10$ (right).   
The dashed lines are the parabolas $2\rf- \rf^2$ and $\rf^2$.  One 
can see that  for larger values of $m$ the   graphs of the 
eigenvalues  concentrate around the graphs of the two parabolas.  
We observe that  in the  one-channel case the eigenvalues   have a 
similar behavior (see Figure~7 in \cite{F}).     
 \begin{figure}[h]
\begin{center}
\centering\setlength{\captionmargin}{0pt} 
 \includegraphics[width=6.5cm,height=5cm]{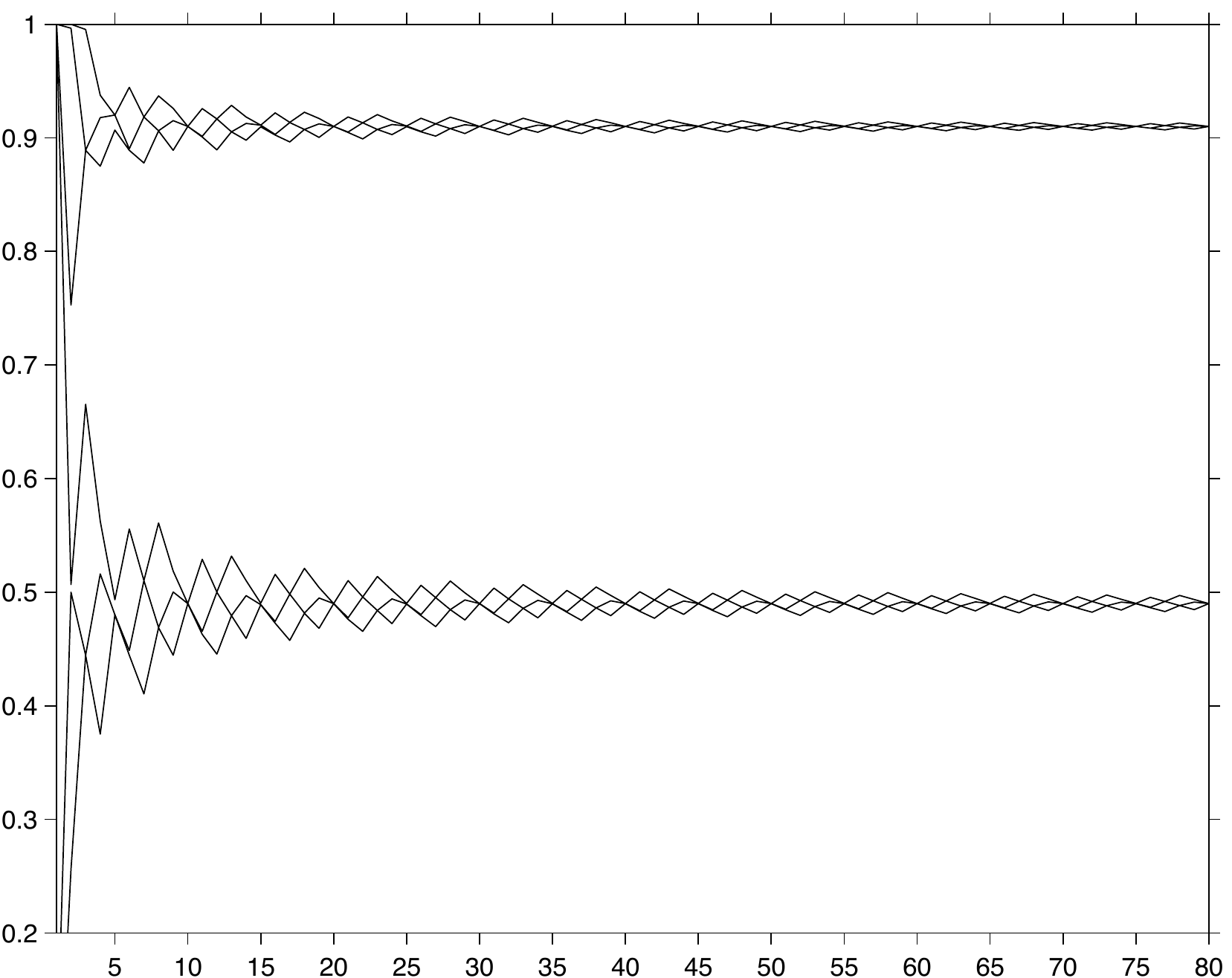}
 %{EigSvrsusmo.pdf} 
  \end{center}
  \vskip-0.1truecm
  \caption{ Some of the eigenvalues of $\Sfi $  as a function of $m $ for $\cI=\{0,1,\dots, 9\}$. $r=.7$ }
 \label{versusm}
\end{figure}  
 \par\noindent  We shall now discuss the behavior of the eigenvalues of the matrix $\Sfi$ as the  parameter $m$ grows. By using  (\ref{matrice})  with   $l_k=m\, i_k$, $k=1,\dots,N$, (\ref{antif}),  (\ref{derdualiferr}) and (\ref{derfi12}),  one can see that,  when  the parameter 
$m$  tends  to infinity, the off-diagonal   entries   tend to zero;    for  each $\epsilon>0$ denote by $m_{\epsilon}$   the positive integer   (depending on 
$N$)  such that, if $m>m_{\epsilon}$,  the off-diagonal  entries of the  matrix  are less 
than $\epsilon$. For such $m_{\epsilon},$ by the  Gershgorin theorem,   $N$ eigenvalues lie in the circle of center $2\rf-\rf^2 $ 
and radius  $\epsilon$  and $N$ eigenvalues lie in the circle of 
center $\rf^2 $ and radius  $\epsilon$.  Figure~\ref{versusm} illustrates this behavior: it shows   the eigenvalues $\lambda_i$ $i=1,5,10,11,15,20$    in dependence of the parameter $m$ for $\rf=.7$ and    $\cI=\{0,1,\dots, 9\}$. Note that the limits of the eigenvalues  are $2\rf-\rf^2$ and $\rf^2$.

  \section{Reconstruction: numerical results} \label{s:Numerical results}  
In this section  we present some numerical experiments on the recovery of missing samples and   on  the reconstruction  of  a signal   via    oversampling formulas. In particular, we analyze   the  case of contiguous missing samples, when   the problem may become very   ill-conditioned, depending on the parameter $\rf $ and the number of missing samples.    Indeed  if $\rf$ is  close to 1, or the number of missing samples is large, the  condition number of the matrix $I-S$ is large, so that  the  system  amplifies  {the errors}  in the data. We solve  this problem, also in presence of noisy  data, by applying a  regularization technique typical of  the treatment of inverse problems (see \cite{BB}, Section 5). These techniques consist in considering a family of approximate solutions $ X_{\lambda}$ depending on a non-negative parameter $\lambda$, called    {\it regularization parameter}.  When the data are noise-free, the solution
$ X_{\lambda}$ converges to the exact solution  as the regularization parameter tends to zero.  {In the case of noisy data,   one can obtain an optimal approximation of the exact solution for a positive   value of the parameter.}\\
Following  Ferreira, in all our experiments we  shall use  the test function  
   \begin{equation}\label{ferreira}
\fer(x)=(\snc(\pi(x-2.1)))-0.7 (\snc(\pi(x+1.7))).
\end{equation}
which has band $[-\pi, \pi]$  \cite{F}. \par\noindent
  First we present some numerical experiments for the recovery of  contiguous  missing samples via the  one-channel formula   (\ref{rico2}). To compute the  sum in (\ref{B}), we must truncate it    to the values $k$ such that $|k|\le M$, for some integer $M$, thus introducing an error.  
In \cite{F} the author   makes the choice $M=40$ to  reconstruct   the signal (\ref{ferreira})     when    $\cU=\{0,1,\dots,5\}$ and $\rf= 0.6$ (hence $t_o=0.6)$. One can verify that   for $M=40$ the   norm of the truncation error is of order  $10^{-3}$ and that    choosing    $M=500$ this error is     reduced   to the order of $10^{-4}$, while, due to round-off errors, a  larger value of $M$  does  not reduce it  further.  In Table~\ref{500} we show the values $x_k=l_k  t_o$ (column 1)  the exact samples  $g(x_k)$ (column 2), the computed samples $g_k^*$   and the relative errors $ e_k^* $ with    $M=40$ (columns 3 and 4),
the computed samples $g_k$ and the relative errors $ e_k $ with    $M=500$ (columns 5 and  6). The improved  results obtained in the latter case  can be explained by observing that   the    condition number of the matrix $I-\S$,   which controls    the propagation of the  errors, is    $  3.07 \cdot 10^4$.
 \begin{table}[h]
\begin{center}
\begin{tabular}{|c|c||c|c||c|c|  }  \hline
\leavevmode
$x_k$ &$g(x_k)$ &  $g_{k}^* $  & $e_k^*$   &$g_k$ &$e_k$ 
 \\   \hline
$0.0$&$ \ \ 0.1529$&     $\ \  0.1132 $ &     $ 0.2598$   &     $ \ \  0.1498  $ &  $   0.0199 $   \\ 
$0.6 $  &$ -0.2906  $&   $ -0.5344$ & $0.8390 $                 & $  -0.3096  $ &  $0.0653 $  \\
 $  1.2 $    & $\ \ 0.0856 $&      $ -0.4833$ & $6.6498 $         & $\ \ 0.0410$  &  $0.5206 $   \\
$ 1.8 $   &  $  \ \ 0.9221$&   $ \ \ 0.2132$ & $0.7687$    &        $ \ \ 0.8664$ & $0.0603$    \\
$  2.4 $	 &  $\ \ 0.8416 $&       $\ \ 0.3498 $ & $ 0.5844$    &  $\ \ 0.8029 $  & $0.0459$  \\
$ 3.0 $	 &  $\ \  0.0710$	&   $-0.0872  $ & $ 2.2288$       &    $\ \  0.0585  $ &  $     0.1751$ \\
\hline
\end{tabular}
\smallskip
\end{center}
\caption{\small{Recovered samples with the one-channel formula,  r=0.6: exact samples (col 2),   computed samples and relative errors with $M=40$  (col.s  3 and 4) and $M=500$   (col.s  5 and 6).}}
\label{500}
\end{table}
\begin{figure}[h]
\begin{center}
\begin{tabular}{l l}
\hskip-1truecm
\includegraphics[width=7.0cm]{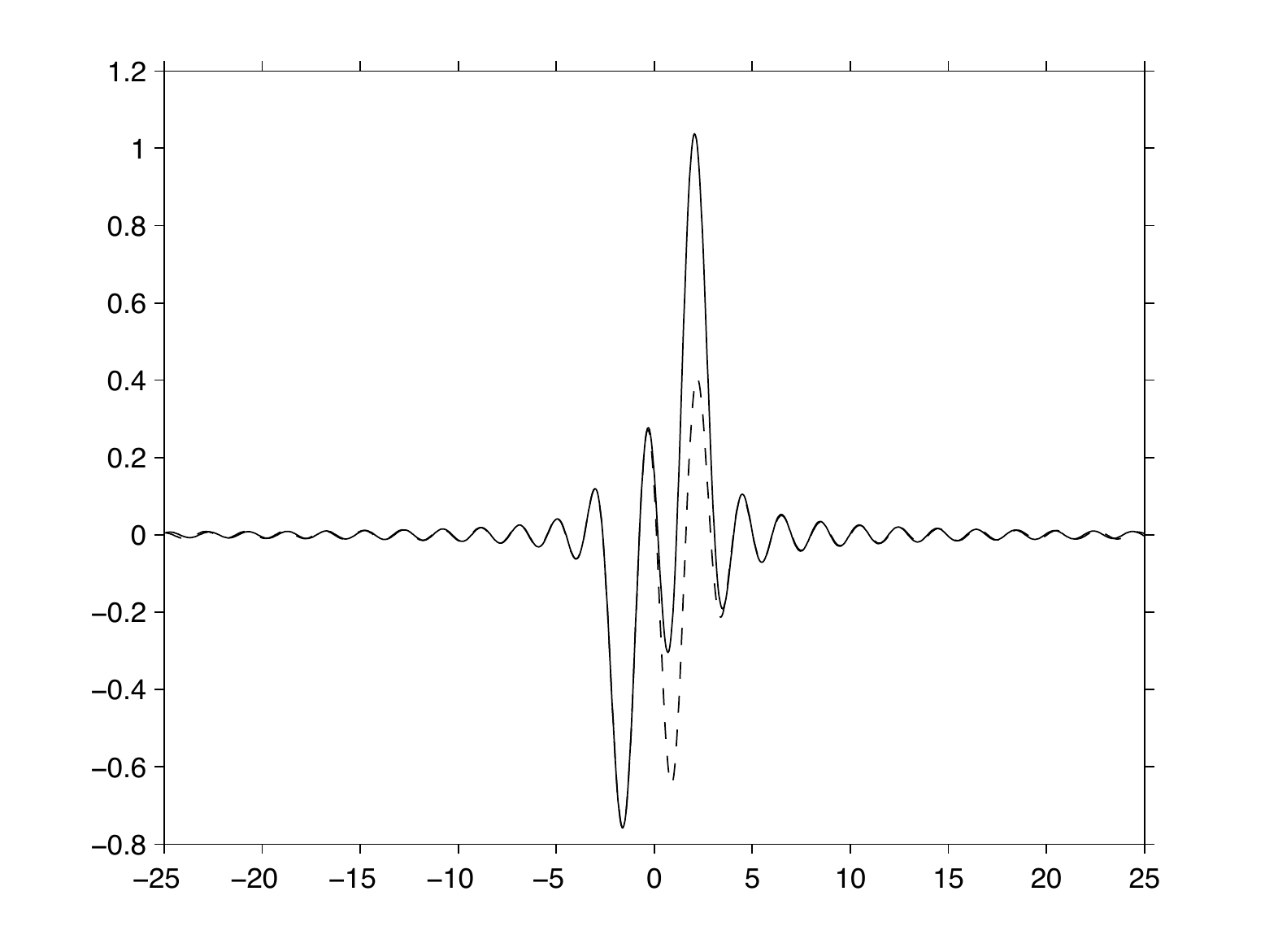}&
\includegraphics[width=7.0cm]{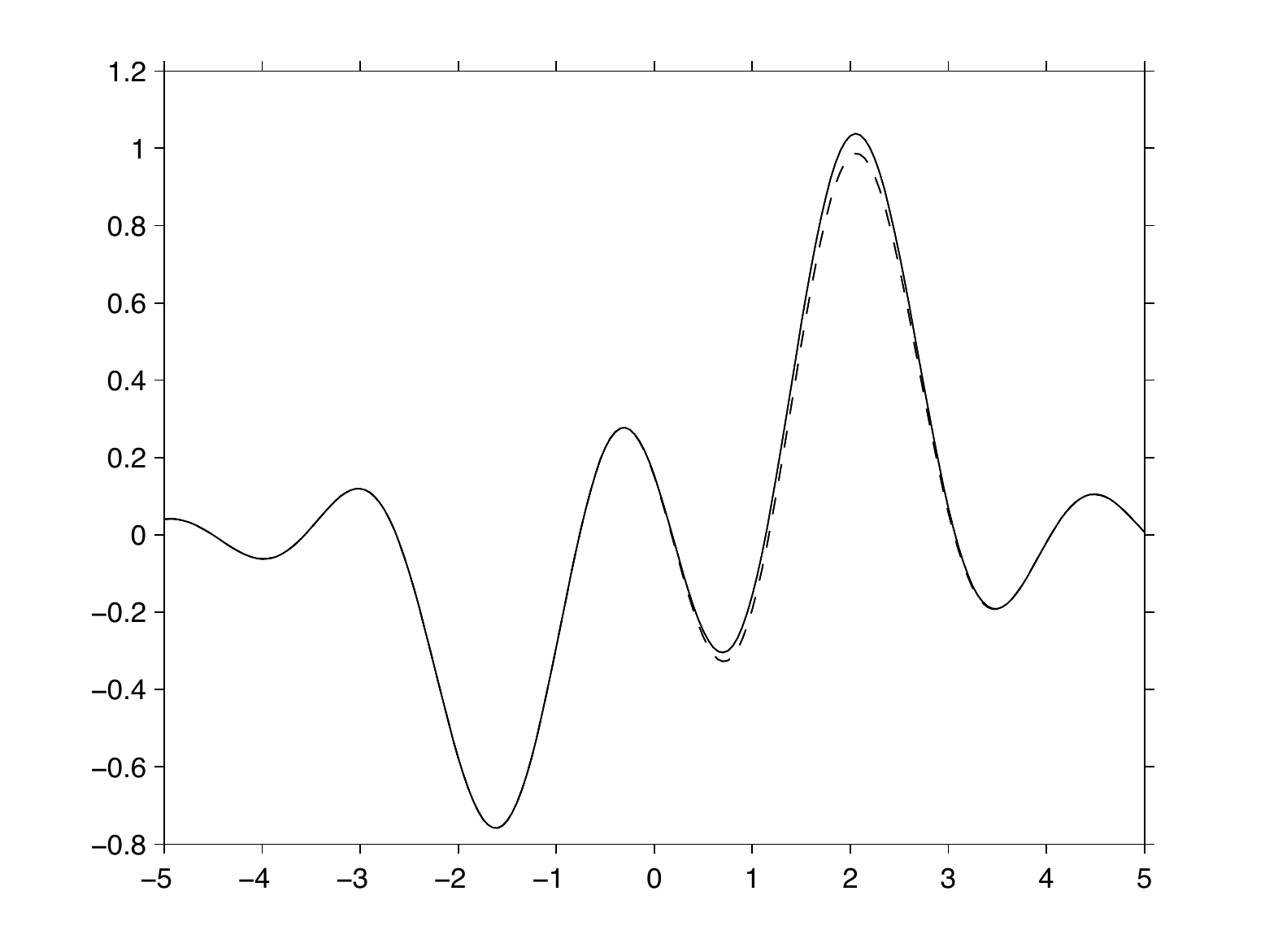}
\end{tabular}
\end{center}
\vskip-.5truecm
 \caption{Original (solid line) and reconstructed (dotted line) signal $g$, $\rf =.6$, $\cU=\{0,1,\dots,5\}$.  Truncation  $M= 40$  (left) and  $M= 500$ (right),    on the interval $[-5,5]$.} 
\label{ferr40V500new}
\end{figure} 
\noindent We have also plotted the graphs: in Figure~\ref{ferr40V500new}  we compare  the original and the reconstructed signal in both cases.  For $M=500$  we  have zoomed  the  graphs, representing them in the interval $[-5,5]$, to make them distinguishable. 
By the preceding considerations,   in  all  the experiments described below, we have chosen $M=500$ to reduce the error due to the truncation error.   \\
Next we consider   the problem of data  affected by noise; we choose  again   $\rf=0.6$  and  
$\cU=\{-2,-1,0,1,2,3\}$.    The  condition number of the matrix $I-\S$ is still  
$3.0769\cdot 10^4$, since  it does not depend  on the position of the 
contiguous missing samples (see (\ref{S1can})). 
  We have introduced  a random noise  of order  $10^{-2}$   on the data and  
   have solved the system (\ref{sistema}) by using the Tikhonov regularization technique.
  Since an estimate of the  norm of the noise is known, it is possible  
to use    the {\it discrepancy  principle}  to  choose   the  value of the 
regularization parameter $\lambda$ (see \cite{BB} and \cite{M}).    Table~\ref{dascrivere} summarizes  the results of this experiment;   columns  1 and 2 contain  the values  $x_k=l_k t_o$ and the exact  values  $g(x_k)$ of the 
 \begin{table}[h]
\begin{center}
\begin{tabular}{|c|c||c|c||c|c| }\hline
\leavevmode
$x_k$ &$g(x_k)$ &$g_k$ &$e_k$&$g^{r}_k$ &$e_k^{r}$ 
 \\   \hline
$-1.2$&$-0.5237$& $\  50.9126 $ &  $ \  \ 98.2227 $ 
& $-0.4852 $&$ 0.0734$ \\ 
$-0.6 $  &$ \ \ 0.1580 $& $ 180.6426  $ &  $1142.5116 $
& $\ \ 0.2141  $ &$0.3550$\\
 $\ \  0.0 $    & $\ \ 0.1529 $& $306.9469 $  &  $2006.8097   $
& $\ \ 0.1430$ &  $0.0652$   \\
$\ \ 0.6 $   &  $ -0.2906$	& $ 
307.9712 $ & $1060.8440$
& $-0.3608$  & $0.2416$   \\
$\ \ 1.2 $	 &  $\ \ 0.0856 $&  $183.2893 $  & $2141.4934$
&$\ \ 0.0435  $ & $0.4918$  \\
$\ \ 1.8 $	 &  $\ \  0.9221 $	& $\ \  53.8130   $ &  $\ \ \  57.3621 $
&$\ \ 0.9248 $  & $  0.0030$\\
\hline
\end{tabular}
\smallskip
\end{center}
\caption{\small{Recovered samples with the one-channel formula,  r=0.6,   with noise added: exact samples (col 2),   computed samples and relative errors without regularization (col.s  3 and 4), computed samples with regularization and the relative errors (col.s 5 and 6).}}
\label{dascrivere}
\end{table}
missing samples,   columns 3 and 4    contain 
  the recovered samples $g_k$  and the relative errors $e_k$ obtained  without regularization, while  columns 5 and 6 show the recovered samples $g^{r}_k$ and the relative errors $e^{r}_k$ obtained via regularization.  
  As in the previous example by comparing  columns 2 and 5, one can see that the absolute errors are very small:    the plots of the real and reconstructed signal would be  indistinguishable.    
\par
We shall now present some experiments  that involve  the two-channel formula.  To compute  the right-hand side of (\ref{sistema}), we have truncated the sums in (\ref{ckck})  retaining only the terms for which    $|n|\le 500$. As in the one-channel case, this leads to an error of order $10^{-4}$.
  Our experiments show that, as in the one-channel case, the recovery    and the reconstruction     
of missing samples of the function and of its  derivative   is rather  efficient  if the  distance between the    missing samples  is  large    and the oversampling parameter $\rf$ is not too close to 1. 
To give an example, for $\cI=\{-1,0,1,2,3,4\}$, $m=4$ and  $\rf=0.7$,    the maximum absolute error   is 
less then $8\cdot10^{-4}$. The error increases   when   the parameter $m$  gets smaller, or   $r$ larger.
 \par Next, we  experiment  the recovery  of   consecutive   missing 
samples   of $g$ and $g'$  when    
$\cU=\{ -2,-1,0,1,2,3 \}$,     $\rf=0.6$ (hence $t_o=1.2$).   The  results are shown in Tables \ref{duecanr6}  and   \ref{duecanr6der}.
 In this case   the condition number of the matrix $I-\Sfi$ in  (\ref{sistema}) is $\ 3.67\cdot10^7,$    much larger than in the one-channel case. This has the effect of propagating   strongly     the truncation error,  thus producing    huge errors  of the solution, much worse than in the one-channel case. Thus,  with this  distribution of missing points and this value of the parameter $\rf$ the system is very  ill-conditioned; an application of   the Tikhonov  regularization  technique  with discrepancy principle leads to  the  results  in columns 5 and 6 in    Table \ref{duecanr6}  and   \ref{duecanr6der}.  In Figure~\ref{last} we show  the plots of the original and the computed  signal $g$.
   \begin{table}[h]
\begin{center}
\begin{tabular}{|c|c||c|c||c|c| }\hline
\leavevmode
$x_k$ &$g(x_k)$ &$g _k$ &$e_k $&$g^{r}_k$ &$e_k^{r}$  \\   \hline
$-2.4$&$-0.1868$& $-0.0580$ &  $ \ -0.6893 $ &          $   -0.0785 $&$  0.5799$ \\ 
$-1.2 $  &$\ -0.5237 $& $\ \ 2.1433   $ &  $-5.0929$&             $  -0.4569 $ &$0.1274$\\
 $\  \  0.0  $    & $\ \  \ 0.1528$& $\ 10.3983$  &  $\ 67.0174   $&       $\ \  0.1832$ &  $0.1989 $   \\
 $\ \ 1.2 $   &  $\ \  \ 0.0856$	& $ 12.0422 $ & $ 139.7624$&       $ -0.1897$  & $3.2170$   \\
 $\ \ 2.4$	 &  $\ \   \  0.8416 $&  $\ \ 5.1810$  & $\ \  \ 5.1561$ &         $  -0.1480  $ & $1.1758$  \\
$\ \ 3.6$	 &  $\ -0.1782 $	& $\ \  0.1425   $ &  $ \, -1.7997 $&$   -0.4405 $  & $ 1.4720$\\
    \hline
\end{tabular}
\smallskip
\end{center}
\caption{\small{Recovered samples of $g$ with two channels,  $r=0.6$,  no noise added: exact samples (col. 2),   computed samples and relative errors (col.s  3 and 4), computed samples  and the relative errors (col.s 5 and 6) with regularization.}}
\label{duecanr6}
\end{table}
 \vskip-0.7truecm
  \begin{table}[h]
\begin{center}
\begin{tabular}{|c|c||c|c||c|c| }\hline
\leavevmode
$x_k$ &$g'(x_k)$ &$g'_k$ &${e'}_k $&$g'^{\,r}_k$ &${e'}_k^{\,r}$  \\   \hline
$-2.4$&$-0.9399$& $-0.4742$ &  $-0.4955$ &$-0.7919 $& $0.1575  $ \\ 
$-1.2 $  &$\ \ 1.0457$& $\ \ 5.5797   $ &  $\ \ 4.3356$&$\ \ 0.8325$& $0.2039$ \\
 $\ \   0.0  $    & $ -0.7350 $& $\ \ 5.4046$  &  $-8.3534 $ &  $ -0.5690 $ & $0.2258$  \\
  $\ \ 1.2 $ &  $\ \ 1.4159$& $ -2.6352    $ & $ \ \ 2.8611$  & $\ \ 0.5867 $ & $0.5857$  \\
 $\ \ 2.4$	 &  $ -1.0603$&  $-7.1265$  & $-5.7212$  & $-0.9327$ & $0.1203 $ \\
$\ \ 3.6$	 &  $\ \ \, 0.2127$& $\, -0.8092   $ &  $\ \ \,  4.8046$ & $\ \  \, 0.8036$ &$2.7786 $ \\
    \hline
\end{tabular}
\smallskip
\end{center}
\caption{\small{Recovered samples of $g'$ with two channels, $ r=0.6$,  no noise added: exact samples (col. 2),   computed samples and relative errors   (col.s  3 and 4), computed samples   and the relative errors (col.s 5 and 6) with regularization.}}
\label{duecanr6der}
\end{table}
Of course,   by using a  smaller oversampling parameter, one   can reduce  the condition number of the matrix $I-\Sfi$ , obtaining  much better results: using  $\rf=0.3$ (hence $t_o=0.6$),  and the same set  $\cU=\{ -2,-1,0,1,2,3 \}$,  one gets $\textrm{cond}(I-\Sfi)=1.92\cdot 10^4 $ and   maximum   absolute error  on the samples of the signal equal to $0.0078$.   
In Section 2  we have observed that,  in the  two-channel formula,    the case    $\rf \in (0,.5)$    implies oversampling in each channel separately, thus  it can be used  the one-channel formula separately for the signal and its derivative. However, the two alternative are not numerically equivalent.  To recover the same   missing samples via  the one-channel formula, we   have  taken    $\rf=0.6$ and $\cU=\{ -2,-1,0,1,2,3 \}$ and we have   obtained   a maximum   absolute error  on the samples of the signal equal to    $0.0274$. This shows that the second alternative  may be  less efficient in case of consecutive samples.
\begin{figure}[h]
\begin{center}
\centering\setlength{\captionmargin}{0pt}%
 \includegraphics[width=7cm,height=5.5cm]{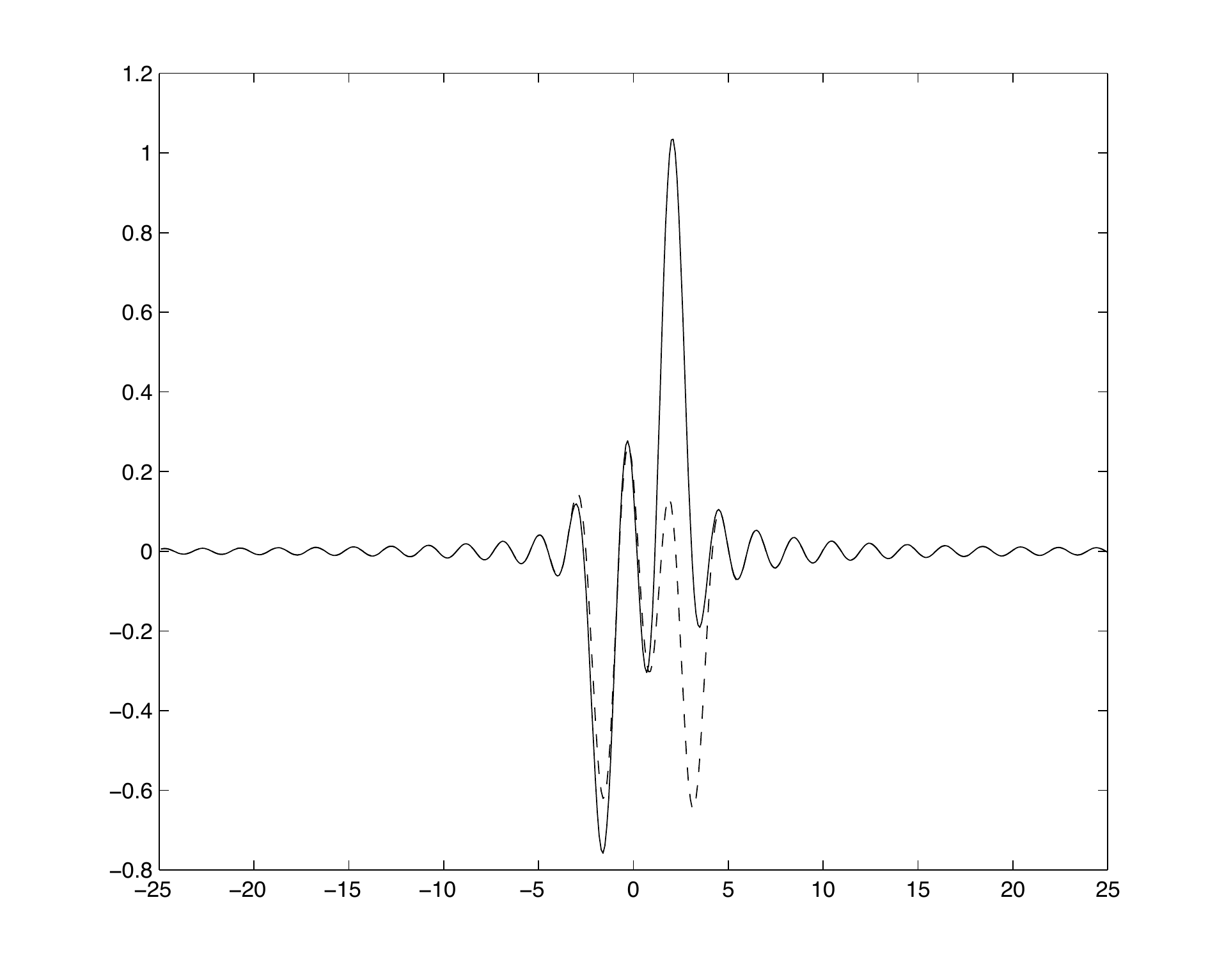}
 %{EigSvrsusmo.pdf} 
  \end{center}
  \vskip-0.4truecm
  \caption{ Original and reconstructed signal $g$, recover with regularization, for $\cU=\{ -2,-1,0,1,2,3 \}$   $r=0.6$. }
 \label{last}
\end{figure}  
%\pagebreak

 \end{document}